\documentclass[a4paper,12pt,reqno]{amsart}
\usepackage{amsmath,amssymb,amsthm}
\usepackage{mathrsfs}
\theoremstyle{plain}
\newtheorem{thm}{Theorem}[section]
\newtheorem{corollary}[thm]{Corollary}

\theoremstyle{definition}
\newtheorem{defn}{Definition}

\theoremstyle{remark}

\numberwithin{equation}{section}
\begin{document}
\newcommand{\ici}[1]{\stackrel{\circ}{#1}}
\begin{center}{\bf{LACUNARY ARITHMETIC STATISTICAL CONVERGENCE}}

\vspace{.5cm}
Taja Yaying$^{1}$ and Bipan Hazarika$^{\ast 2}$ \\
$^{1}$Department of Mathematics, Dera Natung Govt. College, Itanagar-791 111, Arunachal Pradesh, India\\
$^{2}$Department of Mathematics, Rajiv Gandhi University, Rono Hills, Doimukh-791 112, Arunachal Pradesh, India\\
Email: tajayaying20@gmail.com;  bh\_rgu@yahoo.co.in
\end{center}
\title{}
\author{}
\thanks{$^\ast$The corresponding author.}
\date{\today}

\begin{abstract} A lacunary sequence is an increasing integer sequence $\theta=(k_r)$ such that $k_r-k_{r-1}\rightarrow \infty$ as $r\rightarrow \infty.$ In this article we introduce arithmetic statistically convergent sequence space $ASC$ and lacunary arithmetic statistically convergent sequence space $ASC_{\theta}$ and study some inclusion properties between the two spaces. Finally we introduce lacunary arithmetic statistical continuity and establish some interesting results.


AMS Subject Classification No (2010): 40A05, 40A99, 46A70,  46A99.
\end{abstract}
\maketitle
%
\section{Introduction}
A sequence $x=(x_m)$ is called \textit{arithmetically convergent} if for each $\varepsilon > 0$ there is an integer $n$ such that for every integer $m$ we have $\left|x_m- x_{\left\langle m,n\right\rangle}\right|< \varepsilon,$ where the symbol $\left\langle m,n\right\rangle$ denotes the greatest common divisor of two integers $m$ and $n.$ We denote the sequence space of all arithmetic convergent sequence by $AC.$ The idea of arithmetic convergence was introduced by W.H.Ruckle \cite{Ruckle12}. The studies on arithmetic convergence and related results  can be found in \cite{Ruckle12, Tajahazarika, bipan02, tajahazarika3,tajahazacakalli}.\\

By a lacunary sequence we mean an increasing integer sequence $\theta =(k_r)$ such that $h_r=k_r-k_{r-1}\rightarrow \infty$ as $r\rightarrow \infty$. In this paper the intervals determined by $\theta$ will be denoted by $I_{r}=(k_{r-1},k_r]$ and also the ratio $\frac{k_r}{k_{r-1}},r\geq1, k_0\neq 0$ will be denoted by $q_r$. The space of lacunary convergent sequence $N_{\theta}$ was defined by Freedman \cite{FreedmanSemberRaphael} as follows:
\begin{equation*}
N_{\theta}=\left\{x=(x_m)\in w: \lim_{r\rightarrow \infty}\frac{1}{h_r}\sum_{m\in I_r}\left|x_m-l\right|=0 ~\text{for some}~ l\right\}.
\end{equation*}
The space $N_{\theta}$ is a $BK$-space with the norm
\begin{equation*}
\left\|x\right\|_{N_{\theta}}=\sup_r\frac{1}{h_r}\sum_{m\in I_r}\left|x_m\right|.
\end{equation*}
The notion of lacunary convergence has been investigated by \c{C}olak \cite{colak1}, Fridy and Orhan \cite{freedy1,freedy2}, Tripathy and Et \cite{tripathy2} and many others in the recent past.\\

The concept of statistical convergence was introduced by Fast\cite{fast} and later on it was further investigated from the sequence space point of view and linked with summability theory by Fridy \cite{freedy3}, Connor \cite{connor}, Fridy and Orhan \cite{freedy2}, \v{S}al\'at \cite{salat} and many other authors.\\

A sequence $x = (x_m)$ is said to be statistically convergent to the number $L$ if for every $\varepsilon > 0,$
\begin{equation*}
\lim_{t\rightarrow \infty}\frac{1}{t}\left|\left\{m\leq t:\left|x_m-L\right|\geq \varepsilon\right\}\right|=0,
\end{equation*}
where the vertical bars indicate the number of elements in the enclosed set.\\

The main purpose of this paper is to introduce arithmetic statistical convergence and lacunary arithmetic statistical convergence and to study some inclusion properties between these spaces. We also establish some sequential properties of lacunary arithmetic statistical continuity.
\section{Main Results}
A sequence $x=(x_m)$ is said to be arithmetic statistically convergent if for $\varepsilon>0,$ there is an integer $n$ such that
\begin{equation*}
\lim_{t\rightarrow \infty}\frac{1}{t}\left|\left\{m\leq t:\left|x_m-x_{\left\langle m,n\right\rangle}\right|\geq \varepsilon\right\}\right|=0.
\end{equation*} 
We shall use $ASC$ to denote the set of all arithmetic statistical convergent sequences. Thus for $\varepsilon>0$ and integer $n$
\begin{equation*}
ASC=\left\{(x_m):\lim_{t\rightarrow \infty}\frac{1}{t}\left|\left\{m\leq t:\left|x_m-x_{\left\langle m,n\right\rangle}\right|\geq \varepsilon\right\}\right|=0\right\}.
\end{equation*}
We shall write $ASC-\lim x_m=x_{\left\langle m,n\right\rangle}$ to denote the sequence $(x_m)$ is arithmetic statistically convergent to $x_{\left\langle m,n\right\rangle}.$

\begin{thm}
Let $x=(x_m)$ and $y=(y_m)$ be two sequences.
\begin{enumerate}
\item[(i).] If $ASC-\lim x_m=x_{\left\langle m,n\right\rangle}$ and $c\in \mathbb{R},$ then $ASC-\lim cx_m=cx_{\left\langle m,n\right\rangle}.$
\item[(ii).] If $ASC-\lim x_m=x_{\left\langle m,n\right\rangle}$ and $ASC-\lim y_m=y_{\left\langle m,n\right\rangle},$ then $ASC-\lim (x_m+y_m)=(x_{\left\langle m,n\right\rangle}+y_{\left\langle m,n\right\rangle}).$
\end{enumerate}
\end{thm}
\begin{proof}
(i). The result is obvious when $c=0$. Suppose $c\neq 0,$ then for integer $n$
\begin{equation*}
\frac{1}{t}\left|\left\{m\leq t:\left|cx_m-cx_{\left\langle m,n\right\rangle}\right|\geq \varepsilon\right\}\right|=\frac{1}{t}\left|\left\{m\leq t:\left|x_m-x_{\left\langle m,n\right\rangle}\right|\geq \frac{\varepsilon}{\left|c\right|}\right\}\right|
\end{equation*}
which yields the result.\\
The result of (ii) follows from
\begin{align*}
&\frac{1}{t}\left|\left\{m\leq t:\left|(x_m+y_m)-(x_{\left\langle m,n\right\rangle}+y_{\left\langle m,n\right\rangle})\right|\geq \varepsilon\right\}\right|\\
&\leq \frac{1}{t}\left|\left\{m\leq t:\left|x_m-x_{\left\langle m,n\right\rangle}\right|\geq \frac{\varepsilon}{2}\right\}\right|
+ \frac{1}{t}\left|\left\{m\leq t:\left|y_m-y_{\left\langle m,n\right\rangle}\right|\geq \frac{\varepsilon}{2}\right\}\right|.
\end{align*}
\end{proof}
Now we define a related concept of convergence in which the set $\left\{m: m \leq t\right\}$ is replaced by the set $\left\{m: k_{r-1} \leq m \leq k_r\right\},$ for some lacunary sequence $(k_r).$ 
\begin{defn}
Let $\theta=(k_r)$ be a lacunary sequence. The number sequence $x=(x_m)$ is said to be lacunary arithmetic statistically convergent if for each $\varepsilon>0$ there is an integer $n$ such that
\begin{equation*}
\lim_{r\rightarrow \infty}\frac{1}{h_r}\left|\left\{m\in I_r:\left|x_m-x_{\left\langle m,n\right\rangle}\right|\geq \varepsilon\right\}\right|=0.
\end{equation*}
\end{defn}
We shall write
\begin{equation*}
ASC_\theta=\left\{x=(x_m):\lim_{r\rightarrow \infty}\frac{1}{h_r}\left|\left\{m\in I_r:\left|x_m-x_{\left\langle m,n\right\rangle}\right|\geq \varepsilon\right\}\right|=0\right\}.
\end{equation*}
We shall use $ASC_\theta-\lim x_m=x_{\left\langle m,n\right\rangle}$ to denote the sequence $(x_m)$ is lacunary arithmetic statistically convergent to $x_{\left\langle m,n\right\rangle}.$

\begin{thm}
Let $x=(x_m)$ and $y=(y_m)$ be two sequences.
\begin{enumerate}
\item[(i).] If $ASC_\theta-\lim x_m=x_{\left\langle m,n\right\rangle}$ and $c\in \mathbb{R},$ then $ASC_\theta-\lim cx_m=cx_{\left\langle m,n\right\rangle}.$
\item[(ii).] If $ASC_\theta-\lim x_m=x_{\left\langle m,n\right\rangle}$ and $ASC_\theta-\lim y_m=y_{\left\langle m,n\right\rangle},$ then $ASC_\theta-\lim (x_m+y_m)=(x_{\left\langle m,n\right\rangle}+y_{\left\langle m,n\right\rangle}).$
\end{enumerate}
\end{thm}
\begin{proof}
(i). The result is obvious when $c=0$. Suppose $c\neq 0,$ then for integer $n$
\begin{equation*}
\frac{1}{h_r}\left|\left\{m\in I_r:\left|cx_m-cx_{\left\langle m,n\right\rangle}\right|\geq \varepsilon\right\}\right|=\frac{1}{h_r}\left|\left\{m\in I_r:\left|x_m-x_{\left\langle m,n\right\rangle}\right|\geq \frac{\varepsilon}{\left|c\right|}\right\}\right|
\end{equation*}
which yields the result.\\
The result of (ii) follows from
\begin{align*}
&\frac{1}{h_r}\left|\left\{m\in I_r:\left|(x_m+y_m)-(x_{\left\langle m,n\right\rangle}+y_{\left\langle m,n\right\rangle})\right|\geq \varepsilon\right\}\right|\\
&\leq \frac{1}{h_r}\left|\left\{m\in I_r:\left|x_m-x_{\left\langle m,n\right\rangle}\right|\geq \frac{\varepsilon}{2}\right\}\right|
+ \frac{1}{h_r}\left|\left\{m\in I_r:\left|y_m-y_{\left\langle m,n\right\rangle}\right|\geq \frac{\varepsilon}{2}\right\}\right|.
\end{align*}
\end{proof}
\begin{defn}\cite{FreedmanSemberRaphael}
Let $\theta=(k_r)$ be a lacunary sequence. A lacunary refinement of $\theta$ is a lacunary sequence $\theta'=(k'_r)$ satisfying $(k_r)\subseteq (k'_r)$.
\end{defn}
\begin{thm}
\label{fridythm7}
If $\theta'=(k'_r)$ is a lacunary refinement of a lacunary sequence $\theta=(k_r)$ and $(x_m)\in ASC_{\theta'}$ then $(x_m)\in ASC_{\theta}.$ 
\end{thm}
\begin{proof}  
Suppose for each $I_r$ of $\theta$ contains the point $(k'_{r,t})_{t=1}^{\eta(r)}$ of $\theta'$ such that 
\begin{equation*}
k_{r-1}<k'_{r,1}<k'_{r,2}<\ldots <k'_{\eta,\eta(r)}=k_r
\end{equation*} 
where $I'_{r,t}=\left(k'_{r,t-1},k'_{r,t}\right].$\\
Since $(k_r)\subseteq (k'_r),$, so $\forall r,\eta(r)\geq1.$\\
Let $(I^*)_{j=1}^\infty$ be the sequence of interval $(I^*_{r,t})$ ordered by increasing right end points. Since $(x_m)\in ASC_{\theta'},$ then for each $\varepsilon>0$ and an integer $n$
\begin{equation*}
\lim \sum_{I_j^*\subset I_r}\frac{1}{h_j^*}\left|\left\{m\in I_j^*:\left|x_m-x_{\left\langle m,n\right\rangle}\right|\geq \varepsilon\right\}\right|=0.
\end{equation*}  
Also since $h_r=k_r-k_{r-1},$ so $h'_{r,t}=k'_{r,t}-k'_{r,t-1}.$\\
For each $\varepsilon>0$ and integer $n$
\begin{eqnarray*}
\frac{1}{h_r}\left|\left\{m\in I_r:\left|x_m-x_{\left\langle m,n\right\rangle}\right|\geq \varepsilon\right\}\right| &=& \frac{1}{h_r}\sum_{I_j^*\subset I_r}h_j^*\frac{1}{h_j^*}\left|\left\{m\in I_j^*:\left|x_m-x_{\left\langle m,n\right\rangle}\right|\geq \varepsilon\right\}\right|\\
&\rightarrow& 0 ~\text{as $r\rightarrow \infty$}.
\end{eqnarray*}
This implies $(x_m)\in ASC_{\theta}.$
\end{proof}

\begin{thm}
\label{junli}
Suppose $\beta=(l_r)$ is a lacunary refinement of a lacunary sequence $\theta=(k_r).$ Let $I_r=\left(k_{r-1},k_r\right]$ and $J_r=\left(l_{r-1},l_r\right], r=1, 2, \ldots $ If there exists $\delta>0$ such that
\begin{equation*}
\frac{\left|J_j\right|}{\left|I_i\right|}\geq \delta~ \text{for every $J_j\subseteq I_i.$}
\end{equation*}
Then $(x_m)\in ASC_\theta \Rightarrow (x_m)\in ASC_{\beta}.$
\end{thm}

\begin{proof}
For any $\varepsilon>0$ and integer $n$ and every $J_j,$ we can find $I_i$ such that $J_j\subseteq I_i,$ then we have
\begin{align*}
&\left(\frac{1}{\left|J_j\right|}\right) \left|\left\{m\in J_j:\left|x_m-x_{\left\langle m,n\right\rangle}\right|\geq \varepsilon\right\}\right|\\
&= \left(\frac{\left|I_i\right|}{\left|J_j\right|}\right)\left(\frac{1}{\left|I_i\right|}\right) \left|\left\{m\in J_j:\left|x_m-x_{\left\langle m,n\right\rangle}\right|\geq \varepsilon\right\}\right| \\
&\leq \left(\frac{\left|I_i\right|}{\left|J_j\right|}\right)\left(\frac{1}{\left|I_i\right|}\right) \left|\left\{m\in I_i:\left|x_m-x_{\left\langle m,n\right\rangle}\right|\geq \varepsilon\right\}\right| \\
&\leq \frac{1}{\delta}\left(\frac{1}{\left|I_i\right|}\right) \left|\left\{m\in I_i:\left|x_m-x_{\left\langle m,n\right\rangle}\right|\geq \varepsilon\right\}\right|,
\end{align*}
which yields the result.
\end{proof}
\begin{thm}
Suppose $\beta=(l_r)$ and $\theta=(k_r)$ are two lacunary sequences. Let $I_r=\left(k_{r-1},k_r\right], ~J_r=\left(l_{r-1},l_r\right], r=1, 2, \ldots $ and $I_{ij}=I_i\cap J_j,~ i,j=1,2,3\ldots$ If there exists $\delta>0$ such that
\begin{equation*}
\frac{\left|I_{ij}\right|}{\left|I_i\right|}\geq \delta~ \mbox{for~every~}i,j=1,2,3,\dots, ~I_{ij}\neq \phi.
\end{equation*}
Then $(x_m)\in ASC_\theta \Rightarrow (x_m)\in ASC_{\beta}.$
\end{thm}
\begin{proof}
Let $\alpha=\beta\cup \theta.$ Then $\alpha$ is a lacunary refinement of $\theta.$ The interval sequence of $\alpha$ is $\left\{I_{ij}=I_i\cap I_j: I_{ij}\neq \phi\right\}.$ Using theorem \ref{junli} and the condition $\frac{\left|I_{ij}\right|}{\left|I_i\right|}\geq \delta$ yields that $(x_m)\in ASC_\theta \Rightarrow (x_m)\in ASC_{\alpha}.$ Since $\alpha$ is a lacunary refinement of the lacunary sequence $\beta,$ from theorem \ref{fridythm7}, we have $(x_m)\in ASC_\alpha \Rightarrow (x_m)\in ASC_{\beta}.$   
\end{proof}

There is a strong connection between the sequence space $ASC$ and $ASC_\theta.$ Now we give some inclusion relations between the spaces $ASC$ and $ASC_\theta.$
\begin{thm}
\label{lac1}
Let $\theta=(k_r), ~r=1,2,3,\ldots$ be a lacunary sequence. If $\liminf q_r>1,$ then $ASC\subseteq ASC_\theta.$
\end{thm}

\begin{proof}
Let $(x_m)\in ASC$ and $\liminf q_r > 1.$ Then there exists $\sigma>0$ such that $q_r=\frac{k_r}{k_{r-1}} \geq 1+\sigma$ for sufficiently large $r,$ which implies that
\begin{equation*}
\frac{h_r}{k_r}\geq \frac{\sigma}{1+\sigma}. 
\end{equation*}
Then, for sufficiently large $r$ and integer $n,$
\begin{eqnarray*}
\frac{1}{k_r} \left|\left\{m\leq k_r:\left|x_m-x_{\left\langle m,n\right\rangle}\right|\geq \varepsilon\right\}\right| &\geq & \frac{1}{k_r}\left|\left\{m\in I_r:\left|x_m-x_{\left\langle m,n\right\rangle}\right|\geq \varepsilon\right\}\right| \\
&\geq& \frac{\sigma}{1+\sigma}\frac{1}{h_r}\left|\left\{m\in I_r:\left|x_m-x_{\left\langle m,n\right\rangle}\right|\geq \varepsilon\right\}\right| \\
\end{eqnarray*}
Thus $x=(x_m)\in ASC \Rightarrow (x_m)\in ASC_{\theta}.$
\end{proof}

\begin{thm}
\label{lac2}
For $\limsup q_r< \infty$, we have $ASC_{\theta}\subseteq ASC.$
\end{thm}
\begin{proof}
Let $\limsup q_r< \infty$ then there exists $K>0$ such that $q_r< K$ for every $r.$ Let $\tau_r=\left|\left\{m\in I_r:\left|x_m-x_{\left\langle m,n\right\rangle}\right|\geq \varepsilon\right\}\right|$ where $n$ is an integer. Now for $\varepsilon>0$ and $(x_m)\in ASC_{\theta}$ there exists $R$ such that
\begin{equation*}
\frac{\tau_r}{h_r}< \varepsilon \mbox{~for~every~} r\geq R.
\end{equation*}
Let $M=\max\left\{\tau_r: 1\leq r\leq R\right\}$ and let $t$ be any integer with $k_r\geq t\geq k_{r-1}.$ Then for an integer $n$
\begin{eqnarray*}
\frac{1}{t} \left|\left\{m\leq t:\left|x_m-x_{\left\langle m,n\right\rangle}\right|\geq \varepsilon\right\}\right| &\leq & \frac{1}{k_{r-1}}\left|\left\{m\leq k_r:\left|x_m-x_{\left\langle m,n\right\rangle}\right|\geq \varepsilon\right\}\right|\\
&=& \frac{1}{k_{r-1}}\left\{\tau_1+\tau_2+\ldots+\tau_R+\tau_{R+1}+\ldots+\tau_r\right\}\\
&\leq & \frac{MR}{k_{r-1}}+\frac{1}{k_{r-1}}\left\{h_{R+1}\frac{\tau_{R+1}}{h_{R+1}}+\ldots+h_r\frac{\tau_r}{h_r}\right\}\\
&\leq & \frac{MR}{k_{r-1}}+\frac{1}{k_{r-1}}\left(\sup_{r>R}\frac{\tau_r}{h_r}\right)\left\{h_{R+1}+\ldots +h_r\right\}\\
&\leq & \frac{MR}{k_{r-1}}+\varepsilon\frac{k_r-k_R}{k_{r-1}} \\
&\leq & \frac{MR}{k_{r-1}}T+\varepsilon q_r\\
&< & \frac{MR}{k_{r-1}}T+\varepsilon K,
\end{eqnarray*}
which yields $(x_m)\in ASC.$
\end{proof}

\begin{corollary}
From Theorem \ref{lac1} and Theorem \ref{lac2}, if $\theta=(k_r)$ be a lacunary sequence and if
\begin{equation*}
1<\liminf q_r\leq \limsup q_r<\infty,
\end{equation*} 
then $ASC=ASC_\theta.$
\end{corollary}

In \cite{tajabipanlacunary}, Yaying and Hazarika introduced lacunary arithmetic convergent sequence space $AC_{\theta}$ as follows:\\
\begin{equation*}
AC_{\theta}=\left\{(x_m):\lim_{r\rightarrow \infty}\frac{1}{h_r}\sum_{m\in I_r}\left|x_m-x_{\left\langle m,n\right\rangle}\right|=0 ~\text{for an integer}~n \right\}.
\end{equation*} 
We give some relation between the spaces $AC_{\theta}$ and $ASC_{\theta}.$

\begin{thm}
Let $\theta=(k_r)$ be a lacunary sequence; then if $(x_m)\in AC_{\theta}$ then $(x_m)\in ASC_{\theta}.$
\end{thm}
\begin{proof}
Let $(x_m)\in AC_{\theta}$ and $\varepsilon>0.$ We can write, for an integer $n$
\begin{align*}
&\sum_{m\in I_r}\left|x_m-x_{\left\langle m,n\right\rangle}\right|\\
& \geq\sum_{\underset {\left|x_m-x_{\left\langle m,n\right\rangle}\right| \geq \varepsilon}{m\in I_r}}\left|x_m-x_{\left\langle m,n\right\rangle}\right|+  \sum_{\underset {\left|x_m-x_{\left\langle m,n\right\rangle}\right| < \varepsilon}{m\in I_r}}\left|x_m-x_{\left\langle m,n\right\rangle}\right|\\
&\geq \sum_{\underset {\left|x_m-x_{\left\langle m,n\right\rangle}\right| \geq \varepsilon}{m\in I_r}}\left|x_m-x_{\left\langle m,n\right\rangle}\right|\\
&\geq \varepsilon \left|\left\{m\in I_r:\left|x_m-x_{\left\langle m,n\right\rangle}\right|\geq \varepsilon\right\}\right|
\end{align*}
which gives the result.
\end{proof}

\begin{corollary}
In view of \cite[Theorem 2.2]{tajabipanlacunary}, if $(x_m)\in AC$ then $(x_m)\in ASC_{\theta}.$ 
\end{corollary}
\section{Lacunary Arithmetic Statistical Continuity}
In this section we shall introduce lacunary arithmetic statistical continuity and establish some interesting results.
\begin{defn}
A function $f$ defined on a subset $E$ of $\mathbb{R}$ is said to be lacunary arithmetic statistical continuous if it preserves lacunary arithmetic statistical convergence i.e. if $ASC_{\theta}-\lim x_m=x_{\left\langle m,n\right\rangle}$ implies $ASC_{\theta}-\lim f(x_m)=f(x_{\left\langle m,n\right\rangle}).$
\end{defn}
We shall write $ASC_{\theta}$ continuous function to denote lacunary arithmetic statistical continuous function.\\
It is easy to see that the sum and the difference of two $ASC_\theta$ continuous functions is $ASC_\theta$ continuous. Also the composition of two $ASC_\theta$ continuous functions is again $ASC_\theta$ continuous. In the classical case, it is known that the uniform limit of sequentially continuous function is sequentially continuous, now we see that the uniform limit of $ASC_\theta$ continuous functions is also $ASC_\theta$ continuous. 

\begin{thm}
Let $(f_m)_{m\in \mathbb{N}}$ be a sequence of $ASC_\theta$ continuous functions defined on a subset $E$ of $\mathbb{R}$ and $(f_m)$ be uniformly convergent to a function $f,$ then $f$ is $ASC_\theta$ continuous.  
\end{thm}  
\begin{proof}
Let $\varepsilon>0$ and $(x_m)$ be any $ASC_\theta$ convergent sequence on a subset $E$ of $\mathbb{R}.$ By the uniform convergence of $f_m,$ there exist $N\in \mathbb{N}$ such that $\left|f_m(x)-f(x)\right|<\frac{\varepsilon}{3}$ for all $m\geq N$ and for all $x\in E.$\\
Since $f_N$ is $ASC_\theta$ continuous on $E,$ we have for an integer $n$
\begin{equation*}
\lim_{r\rightarrow \infty}\frac{1}{h_r}\left|\left\{m\in I_r:\left|f_N(x_m)-f_N(x_{\left\langle m,n\right\rangle})\right|\geq \frac{\varepsilon}{3}\right\}\right|=0.
\end{equation*}
On the other hand, for an integer $n$ we have
\begin{align*}
&\left\{m\in I_r:\left|f(x_m)-f(x_{\left\langle m,n\right\rangle})\right|\geq \frac{\varepsilon}{3}\right\}\\
&\subset \left\{m\in I_r:\left|f_N(x_{\left\langle m,n\right\rangle})-f(x_{\left\langle m,n\right\rangle})\right|\geq \frac{\varepsilon}{3}\right\}\\
&\cup \left\{m\in I_r:\left|f_N(x_{\left\langle m,n\right\rangle})-f_N(x_m)\right|\geq \frac{\varepsilon}{3}\right\}\\
&\cup  \left\{m\in I_r:\left|f_N(x_m)-f(x_m)\right|\geq \frac{\varepsilon}{3}\right\}. 
\end{align*}
Thus it follows from the above inclusion that
\begin{align*}
&\lim_{r\rightarrow \infty}\frac{1}{h_r}\left|\left\{m\in I_r:\left|f(x_m)-f(x_{\left\langle m,n\right\rangle})\right|\geq \varepsilon\right\}\right|\\
&\leq \lim_{r\rightarrow \infty}\frac{1}{h_r}\left|\left\{m\in I_r:\left|f_N(x_{\left\langle m,n\right\rangle})-f(x_{\left\langle m,n\right\rangle})\right|\geq \frac{\varepsilon}{3}\right\}\right| +\\
& \lim_{r\rightarrow \infty}\frac{1}{h_r}\left|\left\{m\in I_r:\left|f_N(x_{\left\langle m,n\right\rangle})-f_N(x_m)\right|\geq \frac{\varepsilon}{3}\right\}\right|+\\
& \lim_{r\rightarrow \infty}\frac{1}{h_r}\left|\left\{m\in I_r:\left|f_N(x_m)-f(x_m)\right|\geq \frac{\varepsilon}{3}\right\}\right|\\
&=0. 
\end{align*}
Thus $f$ is $ASC_{\theta}$ continuous. 
\end{proof}

\begin{thm}
The set of all $ASC_{\theta}$ continuous functions defined on a subset $E$ of $\mathbb{R}$ is a closed subset of all continuous function on $E,$ i.e. $\overline{ASC_{\theta}(E)}$=$ASC_{\theta}(E),$ where $ASC_{\theta}(E)$ denotes the set of all $ASC_{\theta}$ continuous functions defined on $E$ and $\overline{ASC_{\theta}(E)}$ denotes the closure of $ASC_{\theta}(E).$
\end{thm}
\begin{proof}
Let $f$ be any element of $\overline{ASC_{\theta}(E)}.$ Then there exists a sequence of points in $ASC_{\theta}(E)$ such that $\lim f_m=f.$ Now let $(x_m)$ be any $ASC_\theta$ convergent sequence in $E.$ Since $(f_m)$ converges to $f,$ there exists a positive integer $N$ such that
\begin{equation*}
\left|f(x)-f_m(x)\right|< \frac{\varepsilon}{3},~\forall m\geq N \text{~and~}\forall x\in E. 
\end{equation*}
Now $f_N$ is $ASC_\theta$ continuous on $E$, so we have for an integer $n$
\begin{equation*}
\lim_{r\rightarrow \infty}\frac{1}{h_r}\left|\left\{m\in I_r:\left|f_N(x_m)-f_N(x_{\left\langle m,n\right\rangle})\right|\geq \frac{\varepsilon}{3}\right\}\right|=0.
\end{equation*}
On the other hand, for an integer $n$ we have
\begin{align*}
&\left\{m\in I_r:\left|f(x_m)-f(x_{\left\langle m,n\right\rangle})\right|\geq \frac{\varepsilon}{3}\right\}\\
&\subset \left\{m\in I_r:\left|f_N(x_{\left\langle m,n\right\rangle})-f(x_{\left\langle m,n\right\rangle})\right|\geq \frac{\varepsilon}{3}\right\}\\
&\cup \left\{m\in I_r:\left|f_N(x_{\left\langle m,n\right\rangle})-f_N(x_m)\right|\geq \frac{\varepsilon}{3}\right\}\\
&\cup  \left\{m\in I_r:\left|f_N(x_m)-f(x_m)\right|\geq \frac{\varepsilon}{3}\right\}. 
\end{align*}
From the above inclusion we can write
\begin{align*}
&\lim_{r\rightarrow \infty}\frac{1}{h_r}\left|\left\{m\in I_r:\left|f(x_m)-f(x_{\left\langle m,n\right\rangle})\right|\geq \varepsilon\right\}\right|\\
&\leq \lim_{r\rightarrow \infty}\frac{1}{h_r}\left|\left\{m\in I_r:\left|f_N(x_{\left\langle m,n\right\rangle})-f(x_{\left\langle m,n\right\rangle})\right|\geq \frac{\varepsilon}{3}\right\}\right| +\\
& \lim_{r\rightarrow \infty}\frac{1}{h_r}\left|\left\{m\in I_r:\left|f_N(x_{\left\langle m,n\right\rangle})-f_N(x_m)\right|\geq \frac{\varepsilon}{3}\right\}\right|+\\
 &\lim_{r\rightarrow \infty}\frac{1}{h_r}\left|\left\{m\in I_r:\left|f_N(x_m)-f(x_m)\right|\geq \frac{\varepsilon}{3}\right\}\right|\\
&=0. 
\end{align*}
Thus $f$ is $ASC_{\theta}$ continuous. So $f\in ASC_{\theta}(E)$ which gives us our required result.  
\end{proof}

\end{document}